\documentclass[11pt]{amsart}
\usepackage[top=4.2cm, bottom=4cm, left=2.4cm, right=2.4cm]{geometry}
\usepackage[utf8]{inputenc}
\usepackage[USenglish]{babel}
\usepackage[T1]{fontenc} 
\usepackage{mathrsfs}
\usepackage{mathtools}
\usepackage{amsmath}
\usepackage{amssymb,epsfig}
\usepackage{caption}
\usepackage{subcaption}
\usepackage{comment}
\usepackage{stmaryrd}
\usepackage{todonotes}

\usepackage{amsthm}
\usepackage[absolute]{textpos}
\usepackage{mathtools,emptypage}

\usepackage[bookmarks=true]{hyperref}
\usepackage{xcolor}
\hypersetup{
    colorlinks,
    linkcolor={red!50!black},
    citecolor={blue!50!black},
    urlcolor={blue!80!black},
}

\usepackage{todonotes}
\usepackage[capitalise]{cleveref}
\usepackage{fixme}
\fxusetheme{color}

\newcommand{\N}{\mathbb{N}}

\newcommand{\R}{\mathbb{R}}
\newcommand{\Z}{\mathbb{Z}}

\newcommand{\sfd}{{\sf d}}

\renewcommand{\d}{{\mathrm d}}

\newcommand{\restr}[1]{\lower3pt\hbox{$|_{#1}$}}

\newcommand{\X}{{\rm X}}
\newcommand{\Y}{{\rm Y}}
\renewcommand{\Z}{{\rm Z}}

\newcommand{\Lip}{{\rm Lip}}

\newcommand{\olim}{\omega\textup{-}\lim}

\newcommand{\cM}{\mathbf{M}}
\newcommand{\cN}{\mathbf{N}}

\makeatletter
\newcommand{\mytag}[2]{%
  \text{#1}%
  \@bsphack
  \begingroup
    \@onelevel@sanitize\@currentlabelname
    \edef\@currentlabelname{%
      \expandafter\strip@period\@currentlabelname\relax.\relax\@@@%
    }%
    \protected@write\@auxout{}{%
      \string\newlabel{#2}{%
        {\color{black}#1}%
        {\thepage}%
        {\@currentlabelname}%
        {\@currentHref}{}%
      }%
    }%
  \endgroup
  \@esphack
}
\makeatother

\def\XXint#1#2#3{{\setbox0=\hbox{$#1{#2#3}{\int}$ }
\vcenter{\hbox{$#2#3$ }}\kern-.6\wd0}}

\newcommand{\curr}[1]{\llbracket #1 \rrbracket}
\newcommand{\Mass}[1]{\mathbb{M} #1}

\allowdisplaybreaks

\newtheorem{theorem}{Theorem}[section]

\newtheorem{corollary}[theorem]{Corollary}
\newtheorem{lemma}[theorem]{Lemma}
\newtheorem{proposition}[theorem]{Proposition}
\theoremstyle{definition}
\newtheorem{definition}[theorem]{Definition}

\newtheorem{example}[theorem]{Example}

\newtheorem*{theoremnonumber}{Theorem}

\newcounter{Counter}

\newtheorem{remark}[theorem]{Remark}

\title[A characterization of snowflakes via rectifiability]{A characterization of snowflakes via rectifiability}
\author[Emanuele Caputo]{Emanuele Caputo}\address[Emanuele Caputo]{Mathematics Institute, Zeeman Building, University of Warwick, Coventry, CV4 7AL, United Kingdom}\email{emanuele.caputo@warwick.ac.uk}
\author[Nicola Cavallucci]{Nicola Cavallucci}\address[Nicola Cavallucci]{Département de mathématiques,
Université de Fribourg,
Ch. du musée 23
CH-1700 Fribourg,
Switzerland}\email{n.cavallucci23@gmail.com}
\date{\today}


\begin{document}

\begin{abstract}
    We prove a generalization of Tyson-Wu's characterization of metric spaces biLipschitz equivalent to snowflakes to every metric space, by removing compactness, doubling and embeddability assumptions. We also characterize metric spaces that are biLipschitz equivalent to a snowflake in terms of the absence of non-trivial metric $1$-currents in every ultralimit, or equivalently in terms of purely $1$-unrectifiability of every ultralimit. Finally, we discuss some applications and examples.
\end{abstract}

\maketitle

\section{Introduction}

A metric space $(\X,\sfd)$ is an $
\alpha$-snowflake, $\alpha \in (0,1)$, if there exist a metric space $(\Y,\rho)$ such that $(\X,\sfd)=(\Y,\rho^\alpha)$. We simply call it a snowflake if it is an $\alpha$-snowflake for some $\alpha$. Snowflake metric spaces appear in fractal geometry and are important in the theory of embeddability of metric spaces into finite-dimensional spaces. Assouad embedding theorem says that a snowflake of a doubling metric space biLipschitz embeds in a finite-dimensional space \cite{Assouad1983}. A map $f\colon (\X,\sfd)\to (\Y,\sfd')$ is biLipschitz if there exists $C\ge 1$ such that
$$C^{-1}\sfd(x,x') \le \sfd'(f(x),f(x')) \le C\sfd(x,x')$$
for every $x,x'\in \X$. 
Two metric spaces are biLipschitz equivalent if there exists a surjective biLipschitz map from one into the other. 

Tyson and Wu, inspired by Assouad's result, proved the following characterization of metric spaces that are biLipschitz equivalent to snowflakes, under finite-dimensionality and embeddability assumptions.
\begin{theoremnonumber}[{\cite[Theorem 1.5]{TysonWu2005}}]
\label{thm:TysonWu}
    Let $(\X,\sfd)$ be a compact, metrically doubling metric space that admits a biLipschitz embedding into a uniformly convex Banach space. Then the following conditions are equivalent.
    \begin{itemize}
        \item[(i)] $(\X,\sfd)$ is biLipschitz equivalent to a snowflake.
        \item[(ii)] Every weak tangent of $(\X,\sfd)$ does not contain non-trivial rectifiable curves. 
    \end{itemize}
\end{theoremnonumber}

In their work, a weak tangent of $(\X,\sfd)$ is a pointed Gromov-Hausdorff limit of a sequence of pointed metric spaces $\{(\X,r_n\sfd,x_n)\}_{n\in \N}$ with $r_n > 0$ and $x_n \in \X$ for every $n \in \N$. The metrically doubling assumption is needed in order to ensure that some pointed Gromov-Hausdorff limits exist. On the other hand, the biLipschitz embeddability into some uniformly convex Banach space is used to prove some intermediate condition exploiting the linearity and the convexity properties of the ambient space. However, this embeddability property is not related to any of the equivalent conditions in the thesis.

The first aim of this work is to prove a general version of Tyson-Wu's theorem holding for every metric space. The second goal is to give other characterizations in terms of two classical concepts in geometric measure theory in metric spaces: rectifiability and metric currents. 

\begin{theorem}
\label{thm:main}
    Let $(\X,\sfd)$ be a metric space. Then the following facts are equivalent.
    \begin{itemize}
        \item[(i)] $(\X,\sfd)$ is biLipschitz equivalent to a snowflake.
        \item[(ii)] Every ultralimit of $(\X,\sfd)$ does not contain non-trivial rectifiable curves.
        \item[(iii)] Every ultralimit of $(\X,\sfd)$ is purely $1$-unrectifiable.
        \item[(iv)] Every ultralimit of $(\X,\sfd)$ has no non-trivial metric $1$-currents with inner-regular mass.
    \end{itemize}
\end{theorem}

Ultralimits are generalization of pointed Gromov-Hausdorff limits and apply also in the non-proper setting, see \cite{KapovichDrutu2018}. An advantage of this theory is that ultralimits always exist. Asymptotic cones, tangent cones and ultracompletions of a metric space $(\X,\sfd)$ belong to the class of ultralimits of $(\X,\sfd)$. We refer to Section \ref{subsec:ultralimits} for more details on ultralimits and their relation with pointed Gromov-Hausdorff limits. 

We will prove a slightly stronger version of Theorem \ref{thm:main}, in which one can express conditions (ii)-(iv) in terms of every fixed non-principal ultrafilter. The equivalence between (i) and (ii) is precisely the extension of Tyson-Wu's result to general metric spaces. This is done through other intermediate conditions: the line-fitting property due to T. Laakso and introduced in the appendix of \cite{TysonWu2005} and a new condition on the upper bound on the number of almost equally spaced points that are almost collinear, see Definition \ref{thm:characterization_snowflakes} and Theorem \ref{defin:SRA_equally_spaced}. This condition is inspired by the notion of small rough angles introduced in \cite{DurandCartagenaTyson2025}.
The equivalence between (ii) and (iii) follows by \cite{Bate2022} and the fact that ultralimits of ultralimits are ultralimits, but with respect to a maybe different non-principal ultrafilter. The fact that one can choose a fixed non-principal ultrafilter is a consequence of this principle: every compact subset of an ultralimit is isometric to a Gromov-Hausdorff limit of a sequence of subsets. This is proven and formalized in Proposition \ref{prop:compact_subsets_different_ultralimits} and Remark \ref{rem:GH_to_ultralimits}. The equivalence between (iii) and (iv) follows by a Smirnov-type representation results for metric $1$-currents as a superposition of rectifiable $1$-sets. This is proven independently in \cite{BateCaputoTakacValentineWald2025, ArroyoRabasaBouchitte2025}. The restriction to metric $1$-currents with inner-regular mass is due to the fact that the ultralimits of a metric space are not separable in general.

We discuss some applications of our theory. First, Tyson-Wu presents a counterexample to their main theorem in the case of non-doubling $(\X,\sfd)$. This is a counterexample if we consider pointed Gromov-Hausdorff limits, but not ultralimits. Indeed, we provide an explicit computation that shows that some ultralimit contains a non-trivial geodesic, see Remark \ref{ex:TysonWu}.

Secondly, we prove that, in the class of quasi-selfsimilar spaces, it is enough to check the absence of rectifiable curves or metric currents only in the metric space itself; see Proposition \ref{prop:snowflake_quasi-selfsimilar}. In particular, in view of Proposition \ref{prop:currents_purelyunrect}, this result can be read as follows: in the class of quasi-selfsimilar spaces, the only purely $1$-unrectifiable spaces are biLipschitz equivalent to snowflakes.

A third application shows that a product of two metric spaces is biLipschitz equivalent to a snowflake if and only if both factors are, see Proposition \ref{prop:products}.

In a future work, we will study natural generalizations of conditions (ii)-(iv) of Theorem \ref{thm:main} to higher dimensions, the relations between them and between the metric properties of the space.

\section*{Acknowledgments}
E.C. is supported by the European Union’s Horizon 2020 research
and innovation programme (Grant agreement No. 948021).
We thank J. Mackay for fruitful discussions while the first-named author was in Bristol that led to this work.

\section{Preliminaries}

Let $(\X,\sfd)$ be a metric space. Given $x \in \X$ and $r>0$, we denote by $B(x,r):=\{y \in \X:\, \sfd(y,x)< r\}$ and by $\overline{B}(x,r):=\{y \in \X:\, \sfd(y,x) \le r\}$. A curve is a continuous map $\gamma \colon [0,1] \to \X$.
The length of a curve $\gamma$ is defined as
\begin{equation}
    \label{eq:defin_length_curve}
    \ell(\gamma):=\sup \left\{ \sum_{i=0}^{N-1} \sfd(\gamma({t_i}),\gamma(t_{i+1}))\,:\,0=t_0 < t_1 < \dots < t_N=1,\, N \in \mathbb{N}
    \right\}.
\end{equation}
A curve $\gamma\colon [0,1]\to\X$ such that $\ell(\gamma)<\infty$ is called rectifiable. We say that a curve $\gamma \colon [0,1]\to \X$ is a geodesic if $\sfd(\gamma(t),\gamma(s))=|t-s|\sfd(\gamma(0),\gamma(1))$ for every $t,s \in [0,1]$. In this case, $\ell(\gamma)=\sfd(\gamma(0),\gamma(1))<\infty$.

\begin{remark}
\label{rem:no_rect_in_snowflake}

A snowflake metric space has no non-constant rectifiable curves. This is well-known. For instance, it can be seen as a consequence of the computations in \cite[Example 6.2]{CaputoCavallucci2025III}.
\end{remark}

Every three points $x,y,z$ in an $\alpha$-snowflake metric space $(\X,\sfd^\alpha)$ are quantitatively non-collinear in the following sense:
\begin{equation}
    \label{eq:SRA_condition}
    \sfd(x,y) \le \max\{ \sfd(x,z) + (2^\alpha - 1)\sfd(z,y), (2^\alpha - 1)\sfd(x,z) + \sfd(z,y)\},
\end{equation}
see \cite[Lemma 2.10]{DurandCartagenaTyson2025}. A metric space in which every triple of points satisfy \eqref{eq:SRA_condition} is said to have \emph{rough angles smaller than $2^\alpha -1$}. This notion has been introduced in \cite{DurandCartagenaTyson2025}, inspired by the methods of \cite{LeDonneRajalaWalsberg2018}. We introduce another definition that shares similiarities with the small rough angle condition.

\begin{definition}[Small rough angle for equally-spaced points]
\label{defin:SRA_equally_spaced}
    Let $\varepsilon,\alpha > 0$ and $k\in \N$.
    We say that $(\X,\sfd)$ satisfies the \emph{$\alpha$-small rough angle condition for $(k,\varepsilon)$-equally-spaced points}, and we say that $(\X,\sfd)$ satisfies SRA$_k(\varepsilon,\alpha)$ condition, if the following holds. For every $x_1,\ldots,x_k \in \X$ that satisfy $\max\{\sfd(x_i,x_{i+1}):\, i=1,\dots,k-1\} \le (1+\varepsilon)\min\{\sfd(x_i,x_{i+1}):\, i=1,\dots,k-1\}$, then 
    \begin{equation}
        \label{eq:SRA}
        \sum_{i=1}^{k-1}\sfd(x_{i},x_{i+1}) \ge (1+\alpha)\sfd(x_1,x_k).
    \end{equation}
\end{definition}

\begin{remark}
    The roles of the constants in the above definition are as follows. The integer $k$ refers to the number of points that we consider. The constant $\varepsilon$ expresses how close the $k$ points are from being equally spaced, while the constant $\alpha$ expresses the defect from being collinear.
\end{remark}

\begin{remark}
    Having rough angles smaller than $\beta < 1$ is not a biLipschitz invariant. Indeed, consider in the Euclidean plane the set $\X=\{(-1,0),(0,0),(1,0)\}$. It does not satisfy \eqref{eq:SRA_condition} for any $\beta <1$ when endowed with the Euclidean distance $\sfd_e$. However, the set $\Y=\{(-1,0),(0,y),(1,0)\}$ for $y>0$, endowed with the Euclidean distance, satisfies \eqref{eq:SRA_condition} for $\beta \ge \frac{2}{\sqrt{1+y^2}}-1$ and it is biLipschitz equivalent to $(\X,\sfd_e)$. 
    We will see in Theorem \ref{thm:characterization_snowflakes} that the property
    \begin{equation*}
        (\X,\sfd)\text{ satisfies SRA$_k(\varepsilon,\alpha)$ for some }k \in \N,\varepsilon>0,\alpha>0
    \end{equation*}
    is a biLipschitz invariant.
\end{remark}

Given a metric space $(\X,\sfd)$, a subset $A \subset \X$ and $r>0$ we define $B_\sfd(A,r):=\{ x \in \X:\, \sfd(x,z)< r \text{ for some }z \in A\}$. We need the following definition, which was firstly considered by Tyson-Wu in \cite[Appendix]{TysonWu2005}, after a private communication of the authors with T. Laakso.

\begin{definition}[Line-fitting]
    We say that $(\X,\sfd)$ is \emph{line-fitting} if for every $n\in \N$ there exists a metric $\sfd_n$ on $\X \sqcup [0,1]$ such that
    \begin{itemize}
        \item[(i)] the restriction of $\sfd_n$ on $[0,1]$ is the Euclidean metric;
        \item[(ii)] the restriction of $\sfd_n$ on $\X$ coincides with $c_n\sfd$, for some constant $c_n > 0$;
        \item[(iii)] $[0,1] \subseteq B_{\sfd_n}(\X,\frac{1}{n})$.
    \end{itemize}
\end{definition}

This condition has been used to characterize metric spaces that are biLipschitz equivalent to a snowflake.

\begin{proposition}[{\cite[Theorem 7.2]{TysonWu2005}}]
\label{prop:snowflake_iff_not_line_fitting}
    A metric space $(\X,\sfd)$ is biLipschitz equivalent to a snowflake if and only if it is not line-fitting.
\end{proposition}

The negation of the line-fitting property is also related (actually equivalent as we will see in Theorem \ref{thm:characterization_snowflakes}) to the $\alpha$-small rough angle condition for $(k,\varepsilon)$-equally-spaced points for some $k,\varepsilon,\alpha$.

\begin{lemma}
    \label{lemma:SRA_implies_not_line_fitting}
    Let $(\X,\sfd)$ be a metric space satisfying the \textup{SRA}$_k(\varepsilon,\alpha)$ condition for some $k\in \N$, $\alpha,\varepsilon > 0$. Then it is not line-fitting.
\end{lemma}
\begin{proof}
    Suppose $(\X,\sfd)$ is line-fitting. Choose $n$ big enough and take a metric $\sfd_n$ on $\X \sqcup [0,1]$ that restricts to the Euclidean metric on $[0,1]$, restricts to $c_n\sfd$ on $\X$ and such that $[0,1]$ is contained in the $\frac{1}{n}$-neighbourhood of $\X$. Consider the points $\frac{i}{k} \in [0,1]$ for $i =1,\dots,k$. For every $i$, let $x_i \in \X$ be a point such that $\sfd_n(x_i,\frac{i}{k}) < \frac{1}{n}$. This implies that $\frac{1}{k} - \frac{2}{n} \le \sfd_n(x_i,x_{i+1}) \le \frac{1}{k} + \frac{2}{n}$ for every $i$. Hence, 
    $$\frac{1}{c_n}\left( \frac{1}{k} - \frac{2}{n} \right) \le \sfd(x_i,x_{i+1}) \le \frac{1}{c_n}\left( \frac{1}{k} + \frac{2}{n} \right)$$
    for every $i$. Therefore, 
    $$\frac{\max\{\sfd(x_i,x_{i+1}):\, i=1,\dots,k-1\}}{\min\{\sfd(x_i,x_{i+1}):\, i=1,\dots,k-1\}} \le \frac{\frac{1}{k} + \frac{2}{n}}{\frac{1}{k} - \frac{2}{n}} \le 1+\varepsilon$$
    for $n$ sufficiently large. On the other hand, $c_n\sfd(x_1,x_k) = \sfd_n(x_1,x_k) \ge 1 - \frac{1}{k}-\frac{2}{n}$. Hence,
    $$\frac{\sum_{i=1}^{k-1}\sfd(x_i,x_{i+1})}{\sfd(x_1,x_k)} \le \frac{\sum_{i=1}^{k-1}\left( \frac{1}{k} + \frac{2}{n} \right)}{1- \frac{2}{n}} = \frac{1-\frac{1}{k}+\frac{2(k-1)}{n}}{1-\frac{1}{k} - \frac{2}{n}} \le 1 + \alpha$$
    for $n$ sufficiently large. This contradicts the SRA$_k(\varepsilon,\alpha)$ condition.

\end{proof}

\subsection{Ultralimits}
\label{subsec:ultralimits}

For the proofs of the statements of this section, we refer to \cite{KapovichDrutu2018, CavallucciSambusetti2022, Cavallucci2023}.
A non-principal ultrafilter on $\N$ is a non-trivial finitely-additive measure $\omega$ taking values in $\{0,1\}$ and such that $\omega(A) = 0$ for every finite set $A\subseteq \N$. Non-principal ultrafilters allow to take a well-defined limit of a bounded sequence of real numbers. Given a bounded sequence $\{a_n\}_{n\in \N} \subseteq \R$ and a non-principal ultrafilter $\omega$, then there exists a unique real number $a_\omega$ such that
$$\omega(\{ n \in \N\,:\, \vert a_n - a_\omega \vert < \varepsilon\}) = 1$$
for every $\varepsilon > 0$. The value $a_\omega$ is denoted by $\olim a_n$ and is called the \emph{$\omega$-limit of the sequence $\{a_n\}_{n\in \N}$}. If the sequence is not bounded then only one of the following two possibilities occurs: either $\omega(\{n\in\N \,:\, a_n > M\}) = 1$ for every $M \in \N$ or $\omega(\{n\in\N \,:\, a_n < -M\}) = 1$ for every $M \in \N$. In the first case, we write $\olim a_n = +\infty$ and in the second case $\olim a_n = -\infty$.

The $\omega$-limit satisfies natural properties:
\begin{equation}
\label{eq:ultralimits_algebraic_properties}
    \begin{aligned}
        \olim (a_n + b_n) &= \olim a_n + \olim b_n,\\ \olim (a_n b_n) &= (\olim a_n )(\olim b_n) \\
        \olim f(a_n) &= f(\olim a_n)\\
        \olim a_n &\le \olim b_n,\text{ provided }a_n \le b_n\text{ for }\omega\text{-a.e. }n \in \N
    \end{aligned}
\end{equation}
for every bounded sequences $\{a_n\}_{n\in\N}, \{b_n\}_{n\in\N}$ of real numbers and every continuous function $f\colon \R \to \R$.

A pointed metric space is a triple $(\X,\sfd,x)$, where $(\X,\sfd)$ is a metric space and $x \in \X$. Two pointed metric spaces $(\X,\sfd,x),(\Y,\sfd',y)$ are isometric if there exists a surjective isometry $f \colon \X \to \Y$ such that $f(x)=y$. In this case we write $(\X,\sfd,x) \cong (\Y,\sfd',y)$.

Let $\{(\X_n,\sfd_n,x_n)\}_{n\in \N}$ be a sequence of pointed metric spaces and let $\omega$ be a non-principal ultrafilter. A sequence $\{y_n\}_{n\in \N}$, with $y_n \in \X_n$ for every $n$, is $\omega$-admissible if $\olim \sfd_n(x_n,y_n) < \infty$. Given two $\omega$-admissible sequences $\{y_n\}_{n\in \N}$, $\{z_n\}_{n\in \N}$ we define $\hat{\sfd}_\omega(\{y_n\}_{n\in \N}, \{z_n\}_{n\in \N}) := \olim \sfd_n(y_n,z_n)$. By definition, the $\omega$-limit of the sequence $\{(\X_n,\sfd_n,x_n)\}_{n\in \N}$ is the pointed metric space $(\X_\omega, \sfd_\omega, x_\omega)$, where $\X_\omega$ is the set of equivalence classes of $\omega$-admissible sequences under the equivalence relation $\hat{\sfd}_\omega = 0$, $\sfd_\omega$ is the metric induced by $\hat{\sfd}_\omega$ on $\X_\omega$ and $x_\omega$ is the class of the $\omega$-admissible sequence $\{x_n\}_{n\in \N}$. More generally, given an $\omega$-admissible sequence $\{y_n\}_{n\in \N}$, we denote its class in $\X_\omega$ by $y_\omega$, or by $\olim y_n$. We also write $(\X_\omega,\sfd_\omega, x_\omega) =: \olim (X_n,\sfd_n,x_n)$.

Let $\{(\X_n,\sfd_n,x_n)\}_{n\in \N}$, $\{(\Y_n,\sfd_n',y_n)\}_{n\in \N}$ be two sequences of pointed metric spaces and let $\omega$ be a non-principal ultrafilter. A sequence of maps $f_n \colon \X_n \to \Y_n$ is said to be $\omega$-admissible if $\olim \sfd_n'(f_n(x_n),y_n) < \infty$. If the maps $f_n$ are uniformly $L$-Lipschitz then there is a well-defined map 
$$f_\omega \colon (\X_\omega, \sfd_\omega, x_\omega) \to (\Y_\omega, \sfd_\omega', y_\omega),\,\quad \olim z_n \mapsto \olim f_n(z_n)$$
which is again $L$-Lipschitz. The map $f_\omega$ is called the $\omega$-limit map of the sequence $\{f_n\}_{n\in \N}$.

\begin{lemma}
\label{lemma:properties_of_ultralimits}
Let $\omega$ be a non-principal ultrafilter.
\begin{itemize}
    \item[(i)] Let $\{(\X_n,\sfd_n,x_n)\}_{n\in\N}$ and $\{(\Y_n,\sfd'_n,y_n)\}_{n\in\N}$ be two sequences of pointed metric spaces and let $f_n \colon \X_n \to \Y_n$ be a sequence of $\omega$-admissible uniformly surjective $L$-biLipschitz maps. Then the $\omega$-limit map $f_\omega$ is $L$-biLipschitz and surjective.
    \item[(ii)] Let $\{(\X_n,\sfd_n,x_n)\}_{n\in\N}$ be a sequence of pointed metric spaces and let $\alpha > 0$. Then $\olim (\X_n, \sfd_n^{\alpha}, x_n)$ is isometric as pointed metric space to $(\X_\omega, \sfd_\omega^{\alpha}, x_\omega)$.
\end{itemize}
\end{lemma}

\begin{proof}
    The biLipschitz property of $f_\omega$ can be obtained arguing as we did before the lemma. The surjectivity follows by the definitions and the combination of the fact that $f_n$'s are uniformly $L$-Lipschitz and surjective.
    
    We prove the second statement. A sequence $\{y_n\}_{n\in \N}$ is $\omega$-admissible for $\{(\X_n,\sfd_n^\alpha,x_n)\}_{n\in\N}$, i.e. $\olim \sfd_n^\alpha(x_n,y_n) < \infty$, if and only if $\olim \sfd_n(x_n,y_n) < \infty$, i.e. it is $\omega$-admissible for $\{(\X_n,\sfd_n,x_n)\}_{n\in\N}$. Furthermore, two $\omega$-admissible sequences $\{y_n\}_{n\in \N}$, $\{z_n\}_{n\in \N}$ satisfy $\olim \sfd_n^\alpha(y_n,z_n) = 0$ if and only if $\olim \sfd_n(y_n,z_n) = 0$. This means that there is a well-defined bijection 
    $$g\colon \olim (\X_n, \sfd_n^{\alpha}, x_n) \to \olim (\X_n, \sfd_n, x_n),\quad \olim y_n \mapsto \olim y_n$$
    that sends $\olim x_n$ to $\olim x_n$.
    We denote $\olim (\X_n, \sfd_n, x_n) =: (\X_\omega, \sfd_\omega, x_\omega)$ and we consider the $\alpha$-snowflaked metric $\sfd_\omega^\alpha$. We claim that $g \colon \olim (X_n,\sfd_n^\alpha,x_n) \to (\X_\omega, \sfd_\omega^\alpha, x_\omega)$ is an isometry. Let $\{y_n\}_{n\in\N}$, $\{z_n\}_{n\in\N}$ be two admissible sequences. Then, by \eqref{eq:ultralimits_algebraic_properties} we get that
    $$\olim \sfd_n^\alpha(y_n,z_n) = (\olim \sfd_n(y_n,z_n))^\alpha = \sfd_\omega(y_\omega, z_\omega)^\alpha.$$
    This concludes the proof.
\end{proof}

\begin{definition}[$\omega$-limit of $(\X,\sfd)$]
Let $(\X,\sfd)$ be a metric space and let $\omega$ be a non-principal ultrafilter. An \emph{$\omega$-limit of $(\X,\sfd)$} is the $\omega$-limit of a sequence $\{(\X,r_n\sfd,x_n)\}_{n\in \N}$, where $\{x_n\}_{n\in\N}$ is a sequence of points of $\X$ and $\{r_n\}_{n\in \N}$ is a sequence of positive real numbers.
\end{definition}

When the sequence of basepoints $\{x_n\}_{n\in \N}$ is constant, the $\omega$-limits take special names. For instance, if $\olim r_n = 0$, the corresponding $\omega$-limit is called an asymptotic cone. An asymptotic cone is either proper or not separable by \cite{Sisto2012}. If $\olim r_n = +\infty$, the corresponding $\omega$-limit is called a tangent cone. On the other hand, allowing non-constant but bounded sequences $\{x_n\}_{n\in \N}$, if $\olim r_n = r \in (0,+\infty)$ then the corresponding $\omega$-limit contains an isometrically embedded copy of $(\X,r\sfd)$. In this latter case, if moreover $(\X,\sfd)$ is proper then the corresponding $\omega$-limit is isometric to $(\X,r\sfd)$. If $(\X,\sfd)$ is proper, the set of proper $\omega$-limits coincides exactly with the set of all possible Gromov-Hausdorff limits, see \cite[Section 3]{Cavallucci2023}. 
Within the proof of the next result, we provide a generalization of this fact by showing that every compact subset of an $\omega$-limit is a Gromov-Hausdorff limit.

\begin{proposition}
\label{prop:compact_subsets_different_ultralimits}
    Let $(\X,\sfd)$ be a metric space and let $\omega,\omega'$ be non-principal ultrafilters and let $(\X_\omega, \sfd_\omega, x_\omega)$ be an $\omega$-limit of $(\X,\sfd)$. For every compact subset $\Y \subseteq \X_\omega$, there exists an $\omega'$-limit $(\X_{\omega'},\sfd_{\omega'}, x_\omega')$ of $(\X,\sfd)$ and a compact subset $\Y' \subseteq \X_{\omega'}$ which is isometric to $\Y$.
\end{proposition}

\begin{proof}
    We write explicitly $(\X_\omega, \sfd_\omega, x_\omega) = \olim(\X, r_n\sfd, x_n)$. Since $\Y$ is compact, for every $j \in \N$ there exists a finite $\frac{1}{j}$-separated set $\Y^j_\omega := \{ y^{i,j}_\omega \}_{1\le i \le N_j} \subseteq \Y$ that is maximal. Each point $y_\omega^{i,j}$ is the $\omega$-limit of an $\omega$-admissible sequence $\{y^{i,j}_n\}_{n\in \N}$. We define $\Y^j_n:= \{y^{i,j}_n\}_{1\le i \le N_j} \subseteq (\X,r_n\sfd)$. We check that $\olim (\Y_n^j, r_n\sfd, y_n^{1,j}) \cong (\Y_\omega^j, \sfd_\omega, y_\omega^{1,j})$. Every $\Y_n^j$ has cardinality at most $N_j$. For every $\omega$-admissible sequence $\{z_n\}_{n\in\N}$ with $z_n \in \Y_n^j$ for all $n$ we define the sets $A_i := \{n\in \N\,:\, z_n = y_n^{i,j}\}$ for $i\in\{1,\ldots,N_j\}$. These are a finite number of sets, because $i\in \{1,\ldots, N_j\}$ that are disjoint and whose union is $\N$. Therefore, only one of them has measure $1$, say  $A_{i_0}$. This implies that $\olim z_n = \olim y_{n}^{i_0,j} = y_\omega^{i_0,j} \in \Y_\omega^j$. This shows that $\olim (\Y_n^j, r_n\sfd, y_n^{1,j}) \cong (\Y_\omega^j, \sfd_\omega, y_\omega^{1,j})$. The other inclusion is trivial by construction. 

    By \cite[Proposition 3.11]{Cavallucci2023}, there exists a subsequence $\{n_k\}_{k\in \N}$ such that $(\Y^j_{n_k}, r_{n_k}\sfd, y_{n_k}^{1,j})$ converges in the Gromov-Hausdorff sense to $(\Y^j_\omega, \sfd_\omega, y_\omega^{1,j})$. This is true for every $j\in \N$. Since the Gromov-Hausdorff convergence is metrizable, by a diagonal argument we find a subsequence $(\Z_k, r_{n_k}\sfd, y_{n_k}^{1,j_k}) := (\Y^{j_k}_{n_k}, r_{n_k}\sfd, y_{n_k}^{1,j_k})$ such that $(Z_k,r_{n_k}\sfd,y_{n_k}^{1,j_k})$ converges in the Gromov-Hausdorff to $(\Y, \sfd_\omega, y_\omega)$ for some $y_\omega \in \Y$. Applying \cite[Proposition 3.13]{Cavallucci2023}, we get that $\omega'\textup{-}\lim (Z_k, r_{n_k}\sfd, y_{n_k}^{1,j_k})$ is isometric to $\Y$. Now, $\omega'\textup{-}\lim (Z_k, r_{n_k}\sfd, y_{n_k}^{1,j_k})$ is a subset of $\omega'\textup{-}\lim(\X, r_{n_k}\sfd, x_{n_k}) $, which is an $\omega'$-limit of $(\X,\sfd)$. This concludes the proof.
\end{proof}

\begin{remark}
\label{rem:GH_to_ultralimits}
    The same proof shows two more general statements.
    \begin{itemize}
        \item[1)] Let $\{(\X_n,\sfd_n,x_n)\}_{n\in \N}$ be a sequence of pointed metric spaces and let $\omega,\omega'$ be non-principal ultrafilters. Let $(\X_\omega, \sfd_\omega, x_\omega)$ be the $\omega$-limit of the sequence. For every subset $\Y \subseteq \X_\omega$ which is proper as metric space, there exists a subsequence $\{n_j\}_{j\in \N}$ such that the $\omega'$-limit of the sequence $\{(\X_{n_j},\sfd_{n_j},x_{n_j})\}_{j\in \N}$ contains a subset $\Y'$ which is isometric to $\Y$.
        \item[2)] Let $\{(\X_n,\sfd_n,x_n)\}_{n\in \N}$ be a sequence of pointed metric spaces  which admits a pointed Gromov-Hausdorff limit $(\X,\sfd,x)$ that is proper and let $\omega$ be a non-principal ultrafilter. Then $(\X,\sfd,x)$ is the $\omega$-limit of a pointed sequence $(\Y_n,\sfd_n,x_n)$, where $\Y_n\subseteq \X_n$. In particular, $(\X,\sfd,x)$ isometrically embeds into $(\X_\omega, \sfd_\omega, x_\omega)$.
    \end{itemize} 
\end{remark}

\subsection{Non-standard analysis}
We briefly recall some concepts of non-standard analysis. We will need them only in the special case of real-valued functions. References for this section can be found in \cite{ConleyKechrisTucker-Drob2013, Sayag2022}.

Let $\omega$ be a non-principal ultrafilter. We define $\hat{\R}_\omega := \prod_{n \in \N} \R / \sim_\omega$, where $\{a_n\}_{n\in \N} \sim_\omega \{b_n\}_{n\in \N}$ if and only if $\omega(\{n\in\N\,:\, a_n = b_n\}) = 1$. The equivalence class of a sequence $\{a_n\}_{n\in \N}$ is denoted by $[a_n]$. Given a sequence of subsets $A_n \subseteq \R$, we define 
$$\hat{A}_\omega := \{[a_n]\,:\, \omega(\{n\in \N\,:\, a_n \in A_n\}) = 1\} \subseteq \hat{\R}_\omega.$$
The set $\mathcal{A} := \{\hat{A}_\omega\,:\, A_n \text{ is Borel for every $n\in \N$}\}$ is an algebra of sets of $\hat{\R}_\omega$. The function
$$\hat{\mathcal{L}}^1(\hat{A}_\omega) := \olim \mathcal{L}^1(A_n)$$
defines a finitely-additive measure on $\mathcal{A}$. It can be extended to a unique measure, which we still denote by $\hat{\mathcal{L}}^1$, on the $\sigma$-algebra generated by $\mathcal{A}$, that we call $\Sigma$. 

\begin{proposition}[{Dominated convergence Theorem, \cite[Theorem 2.2.3]{Sayag2022}}]
\label{prop:dominated_convergence}
    Let $I \subseteq \R$ be a compact interval. Let $f_n\colon I \to \R$ be a sequence of $\mathcal{L}^1$-measurable equibounded functions. Define $\hat{f}_\omega \colon \hat{I}_\omega \to \R$ by $\hat{f}_\omega([a_n]) := \olim f_n(a_n)$. Then $\hat{f}_\omega$ is $\Sigma$-measurable, bounded and 
    $$\int_{\hat{I}_\omega} \hat{f}_\omega\,\d\hat{\mathcal{L}^1} = \olim \int_If_n\,\d\mathcal{L}^1.$$
\end{proposition}

\section{Proof of main result}

In this section, we characterize metric spaces that are biLipschitz equivalent to snowflakes.

\begin{theorem}
\label{thm:characterization_snowflakes}
    Let $(\X,\sfd)$ be a metric space and $\omega$ be a non-principal ultrafilter. Then the following are equivalent.
    \begin{itemize}
        \item[(i)] $(\X,\sfd)$  is biLipschitz equivalent to a snowflake.
        \item[(ii)] Every $\omega$-limit of $(\X,\sfd)$ does not contain non-constant rectifiable curves.
        \item[(iii)] Every $\omega$-limit of $(\X,\sfd)$ does not contain non-constant geodesic segments.
        \item[(iv)] There exist $k\in \N$ and $\varepsilon,\alpha > 0$ such that $(\X,\sfd)$ satisfies the \textup{SRA}$_k(\varepsilon,\alpha)$ condition.
        \item[(v)] $(\X,\sfd)$ is not line-fitting.
    \end{itemize}
    In particular, all conditions (i)-(v) are preserved under biLipschitz equivalence.
\end{theorem}

\begin{remark}
    Since conditions (ii) and (iii) are equivalent to properties that do not involve the non-principal ultrafilter $\omega$, then they do not actually depend on the choice of the ultrafilter. This is also consequence of Proposition \ref{prop:compact_subsets_different_ultralimits}, because the image of a rectifiable curve is compact. 
\end{remark}

\begin{remark}
    A priori, it is not clear that conditions (iii) and especially (iv) are biLipschitz invariant. The biLipschitz invariance of (v) comes from Proposition \ref{prop:snowflake_iff_not_line_fitting}, but it can also be proved without difficulties directly from the definition.
\end{remark}

\begin{proof}[Proof of Theorem \ref{thm:characterization_snowflakes}]
    (i)$\Rightarrow$(ii). Suppose that $(\X,\sfd)$ is biLipschitz equivalent to an $\alpha$-snowflake. Then every $\omega$-limit of $(\X,\sfd)$ is biLipschitz equivalent to an $\omega$-limit of an $\alpha$-snowflake by Lemma \ref{lemma:properties_of_ultralimits}.(i), which is itself an $\alpha$-snowflake, by Lemma \ref{lemma:properties_of_ultralimits}.(ii). This implies that every $\omega$-limit of $(\X,\sfd)$ does not contain rectifiable curves, by Remark \ref{rem:no_rect_in_snowflake}.

    (ii)$\Rightarrow$(iii). This is obvious since a geodesic segment is a rectifiable curve.
    
    (iii)$\Rightarrow$(iv). Suppose by contradiction that (iv) does not hold. Then, for every $n\in \N$ we can find points $x^0_n,\ldots x_n^n \in (\X,\sfd)$
    such that
    \begin{equation}
    \label{eq:contradiction_in_SRAk}
    \begin{aligned}
        \max\{\sfd(x^i_n,x^{i+1}_n):\, i=0,\dots,n-1\} &\le \left(1+\frac{1}{n}\right)\min \{\sfd(x^i_n,x^{i+1}_n):\, i=0,\dots,n-1\},\\
        \sum_{i=0}^{n-1}\sfd(x^{i}_n,x^{i+1}_n) &< \left(1 +\frac{1}{n} \right)\sfd(x^{0}_n,x^n_n).\\
    \end{aligned}
    \end{equation}
    We set $r_n := 1/\sfd(x_n^0, x_n^n)$.
    Combining the two conditions in \eqref{eq:contradiction_in_SRAk}, we have that
    \begin{equation}
        \label{eq:inequality_m_i}
        nr_n\sfd(x^i_n,x^{i+1}_n) \le \left(\frac{n+1}{n}\right)^2
    \end{equation}
    for every $i=0,\dots,n-1$, thus, by squaring and summing over $i$  
    \begin{equation}
        \label{eq:inequality_m_i_final}
         n \sum_{i=0}^{n-1} r_n^2\sfd(x^i_n,x^{i+1}_n)^2 \le \left(\frac{n+1}{n}\right)^4.
    \end{equation}
    
    We fix $0< t<s < 1$. Fix $n \ge \max\{\frac{1}{t}, \frac{1}{s-t},\frac{1}{1-s}\}$, so that $\lfloor tn \rfloor \ge 1$, $\lfloor s n \rfloor-\lfloor tn \rfloor\ge 1$ and $n-\lfloor sn \rfloor\ge 1$. We deduce that
    
    {\allowdisplaybreaks[4]
    \begin{align}
        \label{eq:omega_limit_geodesic}
        \nonumber \left( \frac{n+1}{n}\right)^4 &\stackrel{\eqref{eq:inequality_m_i_final}}{\ge}  n\left(\sum_{i=0}^{\lfloor tn \rfloor-1} r_n^2\sfd(x^i_n,x^{i+1}_n)^2+ \sum_{i=\lfloor t n \rfloor}^{\lfloor s n \rfloor-1} r_n^2\sfd(x^i_n,x^{i+1}_n)^2+ \sum_{i=\lfloor sn \rfloor}^{n-1} r_n^2\sfd(x^i_n,x^{i+1}_n)^2\right) \\
        \nonumber& \ge n\frac{1}{\lfloor tn \rfloor} \left(\sum_{i=0}^{\lfloor tn \rfloor -1} r_n\sfd(x^i_n,x^{i+1}_n) \right)^2 +n\frac{1}{\lfloor sn \rfloor - \lfloor tn \rfloor} \left(\sum_{i=\lfloor tn \rfloor}^{\lfloor sn \rfloor - 1} r_n\sfd(x^i_n,x^{i+1}_n) \right)^2  \\ 
        &+ n\frac{1}{n-\lfloor sn \rfloor} \left(\sum_{i=\lfloor sn \rfloor}^{n-1} r_n\sfd(x^i_n,x^{i+1}_n) \right)^2 \\
        \nonumber & \ge \frac{n}{\lfloor tn \rfloor} r_n^2\sfd(x^0_n, x^{\lfloor tn \rfloor}_n)^2 + \frac{n}{\lfloor sn \rfloor - \lfloor tn \rfloor} r_n^2\sfd(x^{\lfloor tn \rfloor}_n,x^{\lfloor sn \rfloor }_n)^2 \\
        \nonumber&+ \frac{n}{n-\lfloor sn \rfloor} r_n^2\sfd(x^{\lfloor sn \rfloor}_n,x_n^n)^2\\
        \nonumber & \ge r_n^2\left( \sfd(x^0_n, x^{\lfloor tn \rfloor}_n)+ \sfd(x^{\lfloor tn \rfloor}_n,x^{\lfloor sn \rfloor }_n)+ \sfd(x^{\lfloor sn \rfloor}_n,x_n^n) \right)^2 \ge 1.
    \end{align}
    }
    
    Here, in the second inequality we used the Cauchy-Schwarz inequality, in the third one the triangular inequality and in the second-to-last one we used that the function $(x,y) \mapsto x^2/y$ is subadditive, because positively $1$-homogeneous and convex. 
    
    We consider $\olim (\X,r_n\sfd,x_n^0) =: (\X_\omega, \sfd_\omega, x_\omega^0)$, which is an $\omega$-limit of $(\X,\sfd)$. We define $\gamma \colon [0,1] \to \X_\omega$, $\gamma(\tau) := \olim x_n^{\lfloor \tau n \rfloor}$ and we claim that $\gamma$ is a non-constant geodesic, contradicting (iii). Passing all the inequalities of \eqref{eq:omega_limit_geodesic} to the $\omega$-limit, we have that the corresponding lines of inequalities in $(\X_\omega, \sfd_\omega)$ are indeed equalities, by the last property in \eqref{eq:ultralimits_algebraic_properties}. In particular, using that $\olim n/\lfloor{tn \rfloor}=1/t$ and $\olim n/\lfloor{sn \rfloor}=1/s$ and the first two properties in \eqref{eq:ultralimits_algebraic_properties} we get that

{\allowdisplaybreaks[4]
\begin{align}
    \label{eq:rigidity1}
    \olim \sum_{i=\lfloor tn \rfloor }^{\lfloor sn \rfloor-1} r_n\sfd(x^i_n,x^{i+1}_n) &= \olim r_n\sfd(x^{\lfloor tn \rfloor}_n,x^{\lfloor sn \rfloor}_n) = \sfd_\omega(\gamma(t),\gamma(s)) 
\end{align}
}
and

\begin{align}
\label{eq:rigidity2}
    \olim (s-t)n \sum_{i=\lfloor tn \rfloor}^{\lfloor sn \rfloor -1} r_n^2\sfd(x^i_n,x^{i+1}_n)^2 &= \olim \left(\sum_{i=\lfloor tn \rfloor}^{\lfloor sn \rfloor -1} r_n\sfd(x^i_n,x^{i+1}_n)\right)^2.
\end{align}
In the case $t=0$ or $s=1$, one obtains \eqref{eq:rigidity1} and \eqref{eq:rigidity2} by repeating verbatim the computations in \eqref{eq:omega_limit_geodesic} by, respectively, neglecting the first and the last terms in the summands.

We define the function $f_n \colon [0,1] \to \R$ as $f_n(r):= nr_n\sfd(x^i_n,x^{i+1}_n)$ if $r \in [\frac{i}{n},\frac{i+1}{n}]$. We compute
\begin{equation*}
    \int_{\frac{\lfloor{tn}\rfloor}{n}}^{\frac{\lfloor{sn}\rfloor}{n}} f_n(r)^2 \,\d r = \sum_{i=\lfloor tn \rfloor}^{\lfloor sn \rfloor-1}n^2 r_n^2\sfd(x_n^i,x_n^{i+1})^2\frac{1}{n} = n \sum_{i=\lfloor tn \rfloor}^{\lfloor sn \rfloor-1} r_n^2\sfd(x_n^i,x_n^{i+1})^2
\end{equation*}
and
\begin{equation*}
    \frac{1}{s-t} \left( \int_{\frac{\lfloor{tn}\rfloor}{n}}^{\frac{\lfloor{sn}\rfloor}{n}} f_n(r)\,\d r \right)^2 = \frac{1}{s-t} \left( \sum_{i=\lfloor tn \rfloor}^{\lfloor sn \rfloor-1} r_n\sfd(x_n^i,x_n^{i+1})\right)^2.
\end{equation*}
This implies, together with \eqref{eq:rigidity2}, that
\begin{equation}
\label{eq:dominated_convergence_limitoutside}
    \olim \frac{1}{s-t} \left( \int_{\frac{\lfloor{tn}\rfloor}{n}}^{\frac{\lfloor{sn}\rfloor}{n}} f_n(r)\,\d r \right)^2 = \olim \int_{\frac{\lfloor{tn}\rfloor}{n}}^{\frac{\lfloor{sn}\rfloor}{n}} f_n(r)^2 \,\d r.
\end{equation}

We observe also that $f_n \le 4$ for every $n\in \N$, by \eqref{eq:inequality_m_i}. We consider the map $\hat{f}_\omega \colon \hat{[0,1]}_\omega \to \R$, $\hat{f}_\omega([a_n]) := \olim f_n(a_n)$ and the non-standard sets $\hat{I}_\omega^{t,s} \subseteq \hat{\R}_\omega$ defined by the sequence $I_n^{t,s} := \{[\frac{\lfloor tn\rfloor}{n}, \frac{\lfloor sn\rfloor}{n}]\}_{n\in \N}$. Proposition \ref{prop:dominated_convergence} applied to the sequence of functions $g_n = f_n \chi_{I^{t,s}_n}$, whose $\omega$-limit is $\hat{f}_\omega \chi_{\hat{I}^{t,s}_\omega}$ and \eqref{eq:dominated_convergence_limitoutside} imply that
\begin{equation}
    \label{eq:equality_Cauchy_Schwarz}
    \frac{1}{s-t} \left(\int_{\hat{I}^{t,s}_\omega} \hat{f}_\omega\,\d \hat{\mathcal{L}^1} \right)^2 = \int_{\hat{I}^{t,s}_\omega}\hat{f}_\omega^2\,\d \hat{\mathcal{L}^1}.
\end{equation}
Observe that 
\begin{equation}
    \label{eq:measure_ultra_interval}
    \hat{\mathcal{L}^1}(\hat{I}^{t,s}_\omega) = \olim \mathcal{L}^1(I_m^{t,s}) = \olim \frac{1}{m}\left(\lfloor sm \rfloor - \lfloor tm \rfloor \right) = s - t.
\end{equation}
Combining \eqref{eq:equality_Cauchy_Schwarz} and \eqref{eq:measure_ultra_interval} and using the Cauchy-Schwarz inequality we conclude that $\hat{f}_\omega$ is $\hat{\mathcal{L}}^1$-a.e. constant on $\hat{I}_\omega^{t,s}$. The choice of $t,s$ is completely arbitrary. Using different values of $t$ and $s$ we conclude that $\hat{f}_\omega$ is $\hat{\mathcal{L}}^1$-a.e. constant on $\hat{[0,1]}_\omega$ and equal to $\Lambda >0$. 

In particular, we have that
\begin{equation*}
\begin{aligned}
\label{eq:from_discrete_to_continuous}
    \sfd_\omega(\gamma(t), \gamma(s)) &\stackrel{\eqref{eq:rigidity1}}{=} \olim \sum_{i=\lfloor tn \rfloor }^{\lfloor sn \rfloor-1} r_n\sfd(x^i_n,x^{i+1}_n) \\
    &=\olim \int_{\frac{\lfloor{tn}\rfloor}{n}}^{\frac{\lfloor{sn}\rfloor}{n}} f_n\,\d \mathcal{L}^1 = \int_{\hat{I}_\omega^{t,s}} \hat{f}_\omega\,\d \hat{\mathcal{L}}^1 = \Lambda|t-s|.
\end{aligned}
\end{equation*}
for every $0 \le t < s \le 1$.
By using $t=0$ and $s=1$, we conclude that $\Lambda=1$.
Therefore, the curve $\gamma$ is a non-constant geodesic in $\X_\omega$.

    (iv)$\Rightarrow$(v). This is Lemma \ref{lemma:SRA_implies_not_line_fitting}.

    (v)$\Rightarrow$(i). This is Proposition \ref{prop:snowflake_iff_not_line_fitting}.
\end{proof}

\section{Absence of currents and snowflakes}

In this section, we characterize the metric spaces biLipschitz equivalent to a snowflake in terms of the triviality of the space of metric $1$-currents on the $\omega$-limits of the space. 

Given a metric space $(\X,\sfd)$, we denote the space of Lipschitz functions from $\X$ into $\R$ with ${\rm Lip}(\X)$ and ${\rm Lip}_b(\X):=\{f \in {\rm Lip}(\X):\,f\text{ is bounded.}\}$. Given $f \in {\rm Lip}(\X)$, we denote by ${\rm LIP}(f)$ its Lipschitz constant. We set $D^1(\X):= \Lip_b(\X)\times \Lip(\X)$. 

The definition of metric currents was given in \cite{AK00} and is as follows.

\begin{definition}[Metric $1$-currents]
    Let $(\X,\sfd)$ be a metric space. We say that a multilinear functional $T\colon D^1(\X) \to \mathbb{R}$ is a \emph{metric $1$-current} provided it satisfies the following axioms:
\begin{itemize}
    \item(Locality) For every $(f,\pi) \in D^1(\X)$, we have $T(f,\pi)=0$ if $\pi$ is constant on a neighborhood of $\{ f \neq 0 \}$.
    \item(Continuity) Given $f\in \Lip_b(\X)$ and $\pi_n,\pi \in \Lip(\X)$ such that
    $\sup_{n \in \mathbb{N}} \Lip (\pi_n) < \infty$ and $\pi_n(x)\to \pi(x)$ for every $x \in \X$, we have
    \begin{equation*}
        \lim_{n \to \infty} T(f,\pi_n)= T(f,\pi).
    \end{equation*}
    \item(Finite-mass condition) There exists a nonnegative finite Borel measure $\mu$ on $\X$ such that
    \begin{equation}
        \label{eq:mass_bound_definition}
        \left|T(f,\pi)\right|\le {\rm LIP}(\pi) \int |f|\,\d \mu.
    \end{equation}
\end{itemize}

The minimal measure $\mu$ that satisfies \eqref{eq:mass_bound_definition} is called the \emph{mass measure} of $T$ and is denoted by $\|T\|$. We call $\Mass(T):=\|T\|(\X)$ the mass of $T$. We denote the space of all metric $1$-currents on $(\X,\sfd)$ by $\cM_1(\X,\sfd)$, or simply by $\cM_1(\X)$, if there is no ambiguity. 
\end{definition}

The \emph{boundary} of a metric $1$-current $T$ is the linear functional $\partial T\colon {\rm Lip}_b(\X) \to \mathbb{R}$ defined as
\begin{equation*}
    \partial T(f):= T(1,f).
\end{equation*}
We say that a metric $1$-current $T$ is \emph{normal}, provided that $\partial T$ admits a nonnegative Borel measure $\nu$ on $\X$ such that
\begin{equation*}
    |\partial T(f)|\le \int |f|\,\d \nu\quad \text{for all }f \in {\rm Lip}_b(\X).
\end{equation*}
We denote the space of all normal $1$-currents on $(\X,\sfd)$ by $\cN_1(\X)$. We also set
\begin{equation*}
\begin{aligned}
    \cM_1^{\rm reg}(\X)&:=\{ T \in \cM_1(\X):\, \|T\|\text{ is inner regular by compact sets}\},\\
    \cN_1^{\rm reg}(\X)&:=\cN_1(\X) \cap \cM_1^{\rm reg}(\X).
\end{aligned}
\end{equation*}

\begin{remark}
We recall that Ambrosio-Kirchheim in \cite{AK00}, based on \cite{Federer69}, worked under the additional assumption that the cardinality of every set is an Ulam number. This axiom is consistent with the standard ZFC set-theory axioms. Under this assumption, every nonnegative Borel measure on $(\X,\sfd)$ is inner regular by compact sets (see \cite[Lemma 2.9]{AK00}). Under this assumption, $\cM_1(\X)= \cM_1^{\rm reg}(\X)$ and $\cN_1(\X)=\cN_1^{\rm reg}(\X)$.
\end{remark}

We give an example of a metric $1$-current. 
Given a metric space $(\X,\sfd)$, a curve fragment is a Lipschitz map $\gamma \colon K \to X$, where $K \subseteq [0,1]$ is compact. We endow the set of all curve fragments by the topology induced by the Hausdorff distance of their graphs in $\R \times \X$. 

Given a curve fragment $\gamma$, we define 
\begin{equation}
\label{eq:associate_current}
    \curr{\gamma}(f,\pi):=\int_K (f \circ \gamma)(t)(\pi \circ \gamma)'(t)\,\d t,
\end{equation}
where $(\pi \circ \gamma)'$ is $\mathcal{L}^1$-a.e.\ well-defined on $K$ by Rademacher's theorem. We have that $\curr{\gamma} \in \cM_1(\X)$. If $K=[0,1]$, then $\curr{\gamma} \in \cN_1(\X)$.

We recall that a metric space $(\X,\sfd)$ is called $1$-rectifiable if it can be covered, up to $\mathcal{H}^1$-null set, by Lipschitz images of Borel subsets of $\R$. A subset $E\subset \X$ of a metric space is $1$-rectifiable if $(E,\sfd)$ is $1$-rectifiable. A metric space $(\X,\sfd)$ is purely $1$-unrectifiable if it does not contain any rectifiable $1$-set $E$ such that $\mathcal{H}^1(E)>0$.

\begin{remark}
\label{rem:using_area_formula}
We refer the reader for instance to \cite[Section 2]{BateCaputoTakacValentineWald2025} for the definition of metric speed of a curve fragment. Given a curve fragment $\gamma \colon K \to \X$, it follows by the definition of mass that $\Mass(\curr{\gamma})\le \int_K |\dot{\gamma}_t|\,\d t$. Hence, if $\Mass(\curr{\gamma}) > 0$, then $|\dot{\gamma}_t|>0$ on a set of positive Lebesgue measure in $K$, then by \cite[Theorem 7]{Kirchheim1994}, $\mathcal{H}^1({\rm Im}(\gamma))>0$, so the image of $\gamma$ is a non-trivial $1$-rectifiable set. 
\end{remark}

We now prove Theorem \ref{thm:main} in the following generalized formulation.

\begin{theorem}
\label{thm:snowflake_currents}
    Let $(\X,\sfd)$ be a metric space. Let $\omega$ be a non-principal ultrafilter. The following are equivalent.
    \begin{itemize}
        \item[(i)] $(\X,\sfd)$ is biLipschitz equivalent to a snowflake.
        \item[(ii)] Every $\omega$-limit of $(\X,\sfd)$ does not contain non-constant rectifiable curves.
        \item[(iii)] Every $\omega$-limit of $(\X,\sfd)$ does not contain non-trivial $1$-rectifiable sets, i.e.\ is purely $1$-unrectifiable.
        \item[(iv)] Every $\omega$-limit $(\X_\omega, \sfd_\omega, x_\omega)$ of $(\X,\sfd)$ has $\cN_1^{\rm reg}(\X_\omega)=\{0\}$.
        \item[(v)] Every $\omega$-limit $(\X_\omega, \sfd_\omega, x_\omega)$ of $(\X,\sfd)$ has $\cM_1^{\rm reg}(\X_\omega)=\{0\}$.
    \end{itemize}
\end{theorem}

We recall that every $\omega$-limit is a complete metric space, see \cite[Proposition 7.44]{KapovichDrutu2018}.

\begin{proof}
(i)$\Leftrightarrow$(ii) is proved in Theorem \ref{thm:characterization_snowflakes}.

(ii)$\Rightarrow$(iii) By contradiction, suppose that there exists an $\omega$-limit $(\X_\omega, \sfd_\omega, x_\omega)$ of $(\X,\sfd)$ containing a non-trivial $1$-rectifiable set $E$. By the proof of \cite[Theorem 6.6]{Bate2022}, there exist a closed subset $C \subseteq E$ and a point $x\in C$ that has a pointed Gromov-Hausdorff tangent isometric to $\R$. 
By item 2) in Remark \ref{rem:GH_to_ultralimits}, $(\X_\omega,\sfd_\omega)$ has an $\omega$-limit which contains $\R$. The $\omega$-limit of an $\omega$-limit is isometric to an $\omega'$-limit for some other non-principal ultrafilter $\omega'$, by \cite[Proposition 7.64]{KapovichDrutu2018}. This implies that there exists an $\omega$'-limit of $(\X,\sfd)$ that contains a geodesic segment. Hence, by Proposition \ref{prop:compact_subsets_different_ultralimits} or Theorem \ref{thm:characterization_snowflakes}, some $\omega$-limit of $(\X,\sfd)$ contains a geodesic segment, contradicting (ii).

(iii)$\Rightarrow$(v). By contradiction, there exist an $\omega$-limit $(\X_\omega,\sfd_\omega, x_\omega)$ and $T\in \cM_1^{\rm reg}(\X_\omega,\sfd_\omega)$ with $T \neq 0$. Since $\|T\|$ is inner regular by compact sets, there exists a $\sigma$-compact set $E\subseteq \X_\omega$ such that $\|T\|(\X_\omega\setminus E)=0$. We consider $\Y:=\bar{E}$; we have $\|T\|(\X_\omega\setminus \Y)=0$.

By \cite[Proposition B.4]{BateCaputoTakacValentineWald2025}, which can be applied since any $\omega$-limit is complete, there exists a metric $1$-current $\tilde{T} \in \cM_1(\Y,\sfd_\omega)$ such that $\Mass(\tilde{T})=\Mass(T)\neq 0$. We apply \cite[Theorem 1.4]{BateCaputoTakacValentineWald2025} to $(\Y,\sfd_\omega)$, which is complete and separable. Then there exists a finite non-negative Borel measure $\eta$ on curve fragments of $\Y$ such that $\tilde{T}=\int \curr{\gamma}\,\d \eta(\gamma)$ and
\begin{equation}
    \label{eq:mass_condition}
    \Mass(\tilde{T})=\int \Mass(\curr{\gamma})\,\d \eta(\gamma).
\end{equation}
Since $\Mass(\tilde{T})\neq 0$, the identity \eqref{eq:mass_condition} implies that there exists at least a curve fragment $\gamma\colon K \to \Y$ such that $\Mass(\curr{\gamma})\neq 0$, thus its image is a non-trivial $1$-rectifiable set, by Remark \ref{rem:using_area_formula}.

(v)$\Rightarrow$(iv) is trivial because $\cN_1^{\rm reg}(\X_\omega)\subseteq \cM_1^{\rm reg}(\X_\omega)$. 

(iv)$\Rightarrow$(ii). Every non-constant rectifiable curve $\gamma \colon [0,1]\to \X_\omega$ has a Lipschitz reparametrization $\tilde{\gamma}$, that induces the non-trivial normal $1$-current $\curr{\tilde{\gamma}}$.
\end{proof}


If the $\omega$-limit is a proper metric space (which is equivalent to separable for asymptotic cones by \cite{Sisto2012}), so in particular it is complete and separable, every metric $1$-current has mass that is inner regular by compact sets. The properness of every $\omega$-limit of $(\X,\sfd)$ is equivalent to $(\X,\sfd)$ being metrically doubling: there exists $D\in \N$ such that for every $r\ge 0$ it is possible to cover every ball of radius $2r$ with at most $D$ balls of radius $r$. In this case, every $\omega$-limit is indeed a Gromov-Hausdorff limit.

\begin{corollary}
\label{cor:trivial_currents_doubling}
    Let $(\X,\sfd)$ be a metrically doubling metric space. Let $\omega$ be a non-principal ultrafilter. The following are equivalent.
    \begin{itemize}
        \item[(i)] $(\X,\sfd)$ is biLipschitz equivalent to a snowflake.
        \item[(ii)] Every $\omega$-limit $(\X_\omega, \sfd_\omega, x_\omega)$ of $(\X,\sfd)$ has \ $\cN_1(\X_\omega)=\{0\}$.
        \item[(iii)] Every $\omega$-limit $(\X_\omega, \sfd_\omega, x_\omega)$ of $(\X,\sfd)$ has \ $\cM_1(\X_\omega)=\{0\}$.
    \end{itemize}
\end{corollary}

In Theorem \ref{thm:snowflake_currents} and Corollary \ref{cor:trivial_currents_doubling} it is necessary to ask for the triviality of metric $1$-currents on every $\omega$-limit of $(\X,\sfd)$. Indeed, the set $\N$ equipped with the Euclidean metric is an example of a metric space with trivial metric $1$-currents but with an $\omega$-limit that contains a geodesic. 

The proof of Theorem \ref{thm:snowflake_currents} also shows the following proposition.

\begin{proposition}
\label{prop:currents_purelyunrect}
    A complete metric space $(\X,\sfd)$ is purely $1$-unrectifiable if and only if $\cM_1^{\rm reg}(\X) = \{0\}$.
    Moreover, a complete and separable metric space $(\X,\sfd)$ is purely $1$-unrectifiable if and only if $\cM_1(\X) = \{0\}$.
\end{proposition}

\section{Applications}

In this section we view some examples of spaces that are or are not biLipschitz equivalent to snowflakes.

\begin{remark}[Tyson-Wu's example]
    \label{ex:TysonWu}
    The authors in \cite[Example 5.13]{TysonWu2005} provide an example of a metric space that is neither doubling, nor compact, it is not biLipschitz equivalent to a snowflake but such that every weak tangent in the Gromov-Hausdorff sense has no rectifiable curves. This is not in contrast with Theorem \ref{thm:characterization_snowflakes}, as we show, given any non-principal ultrafilter $\omega$, an explicit $\omega$-limit of Tyson-Wu's example that contains a non-trivial geodesic. 
    
    For every $n\in \N$ we define the set 
    \begin{equation*}
        \X_n = Q_n \cap 2^{-n}\mathbb{Z}^n \subseteq \R^n\text{, where }Q_n = \{ z\in \R^n : \|z-2e_1^n\|_\infty \le 2 n^{-\frac{1}{2}}\} \subseteq \R^n.
    \end{equation*}
    Here, $e_1^n$ is the vector $(1,0,\dots,0) \in\R^n$. For every $n \in \N$, we define $i_n \colon \R^n \to \ell^2$ as
    \begin{equation*}
        i_n(x_1,\dots, x_n):=\sum_{i=1}^n x_i \hat{e}\left(\frac{n(n-1)}{2} + i\right),
    \end{equation*}
    where $\{\hat{e}(j)\}_{j\in \N}$ is an orthonormal basis of $\ell^2$. Tyson-Wu's example is $\X := \bigcup_{n\in \N} i_n(\X_n)$, where the metric $\sfd$ is the restriction of the metric of $\ell^2$. We set $x_n:= i_n(2e_1^n) \in \X$ and $r_n:=2^{-1} n^\frac{1}{2}$. Let $\omega$ be a non-principal ultrafilter. We consider the $\omega$-limit of $(\X,\sfd)$ given by $\olim(\X,r_n\sfd,x_n) = (\X_\omega, \sfd_\omega, x_\omega)$. 
    
    We claim that $(\X_\omega, \sfd_\omega, x_\omega)$ is isometric, as pointed metric space, to $\olim ([-1,1]^n,\sfd_e,0_n)$, where $\sfd_e$ and $o_n$ respectively stands for the Euclidean distance and the origin in $\R^n$. Let $y_\omega = \olim y_n \in \X_\omega$, so $\olim r_n\sfd(x_n,y_n) < \infty$. We observe, that for $\omega$-a.e.$n$, the point $y_n$ belongs to $i_n(\X_n)$. We define the points $z_n := r_n(i_n^{-1}(y_n) - 2e_1^n) \in [-1,1]^n$. The sequence $z_n$ is admissible in $([-1,1]^n, \sfd_e, 0_n)$. This allows to define the map
    $$f\colon (\X_\omega, \sfd_\omega, x_\omega) \to \olim ([-1,1]^n,\sfd_e,0_n), \qquad y_\omega \mapsto z_\omega.$$
    By construction, $f$ is an isometric embedding. Let $z_\omega = \olim z_n \in  \olim ([-1,1]^n,\sfd_e,0_n)$. For every $n$ we define $y_n$ as the closest point to $i_n(\frac{1}{r_n}z_n + 2e_1^n)$ in $i_n(\X_n)$. Observe that $\| y_n - i_n(\frac{1}{r_n}z_n + 2e_1^n) \|_{\ell^2} < 2^{-n}$. By triangular inequality, the sequence $\{y_n\}_{n\in \N}$ is admissible, since $\olim 2^{-n}r_n = 0$. Moreover, by construction and triangular inequality, $f(\olim y_n) = z_\omega$, i.e. $f$ is surjective. This proves the claim.
    The space $\olim ([-1,1]^n,\sfd_e,0_n)$ contains a non-trivial geodesic, because ultralimit of geodesics is a geodesic.



    

\end{remark}

In some cases, it is enough to check the equivalent properties of Theorem \ref{thm:snowflake_currents} only for the metric space, and not for every possible ultralimit. For instance, this is the case for compact quasi-selfsimilar spaces.

\begin{definition}
\label{defin:quasi-selfsimilar}
    Let $\rho_0>0$ and $L_0\geq 1$. A metric space $(\X,\sfd)$ is $(L_0,\rho_0)$-quasi-selfsimilar if for every $\overline{B}(x,\rho)\subseteq \X$ with $0<\rho \leq \rho_0$ there is a map $\Phi\colon (\overline{B}(x,\rho), \frac{\rho_0}{\rho}\sfd) \to \X$ which is $L_0$-biLipschitz.
\end{definition}

This is the definition of quasi-selfsimilar spaces given in \cite{Sullivan82}. In other papers the condition is sometimes more restrictive, see \cite{Kleiner2006, Cavallucci2025}, while other times the biLipschitz condition is weakened to a quasisymmetric one, see \cite{CarrascoPiaggio2011, ErikssonBique2023preprint}.

\begin{proposition}
\label{prop:snowflake_quasi-selfsimilar}
    Let $(\X,\sfd)$ be a compact, $(L_0,\rho_0)$-quasi-selfsimilar metric space. Then the following are equivalent.
    \begin{itemize}
        \item[(i)] $(\X,\sfd)$ is biLipschitz equivalent to a snowflake.
        \item[(ii)] $(\X,\sfd)$ does not contain non-trivial rectifiable curves.
        \item[(iii)] $\cM_1(\X) = \{0\}$.
    \end{itemize}
\end{proposition}

\begin{proof}
    The implication (i) $\Rightarrow$ (iii) follows by Theorem \ref{thm:snowflake_currents} because $(\X,\sfd)$ is isometric to one of its $\omega$-limits and using $\cM_1^{\rm reg}(\X)=\cM_1(\X)$, because $(\X,\sfd)$ is complete and separable. The implication (iii) $\Rightarrow$ (ii) is obvious, because to a rectifiable curve we can associate the metric $1$-current defined in \eqref{eq:associate_current}. 
    To prove that (ii) implies (i) it is enough to show that if $(\X,\sfd)$ does not contain rectifiable curves then also every of its ultralimit does not contain rectifiable curves, by Theorem \ref{thm:snowflake_currents}.
    Let $\omega$ be a non-principal ultrafilter and let $(\X_\omega, \sfd_\omega, x_\omega) = \olim(\X,r_n\sfd,x_n)$ be an arbitrary $\omega$-limit of $(\X,\sfd)$. 
    
    If $\olim r_n = 0$, then $\X_\omega = \{x_\omega\}$ because the diameter of $(\X,\sfd)$ is finite; therefore, the $\omega$-limit of the diameters of $(\X,r_n\sfd)$ is equal to $0$. We conclude that $(\X_\omega, \sfd_\omega)$ does not contain non-trivial rectifiable curves in this case.

    If $\olim r_n = r \in (0,+\infty)$, then $(\X_\omega, \sfd_\omega)$ is isometric to $(\X, r\sfd)$ due to the compactness of $(\X,\sfd)$. Therefore, $(\X_\omega, \sfd_\omega)$ does not contain rectifiable curves, since $(\X,\sfd)$ does not.

    It remains the case $\olim r_n = +\infty$. Suppose that $(\X_\omega, \sfd_\omega)$ contains a non-trivial rectifiable curve $\gamma$. We can suppose that $\gamma$ is contained in $B(x_\omega, R)$ for some $R > 0$. We set $\rho_n := Rr_n^{-1}$.
    For $\omega$-a.e.\ $n$ we have $\rho_n \le \rho_0$. Thus, we find $L_0$-biLipschitz maps $\Phi_n\colon (\overline{B}(x_n,\rho_n), \frac{\rho_0}{R}r_n \sfd) \to (\X,\sfd)$. 
    By definition, there is an isometric embedding $\iota \colon (B(x_\omega, R),\sfd_\omega, x_\omega) \to \olim(\overline{B}(x_n,\rho_n),r_n\sfd,x_n)$,    
    with $B(x_\omega,R) \subseteq \X_\omega$. Hence, we can find an isometric embedding
    $$\iota' \colon \left(B(x_\omega, \rho_0), \frac{\rho_0}{R}\sfd_\omega, x_\omega\right) \to  \olim\left(\overline{B}(x_n,\rho_n),\frac{\rho_0}{R}r_n\sfd,x_n\right).$$
    Clearly, $(B(x_\omega, \rho_0),\frac{\rho_0}{R}\sfd_\omega)$ contains a non-trivial rectifiable curve.
    
    By Lemma \ref{lemma:properties_of_ultralimits}, $\olim \left(\Phi_n(\overline{B}(x_n,\rho_n)), \sfd, \Phi_n(x_n)\right)$ contains a non-trivial rectifiable curve. But $\Phi_n(\overline{B}(x_n,\rho_n)) \subseteq \X$ for every $n$, so the $\omega$-limit above is contained in $\olim (\X,\sfd,\Phi_n(x_n))$. By compactness of $(\X,\sfd)$, this $\omega$-limit is actually isometric to $(\X,\sfd,z)$ for some $z\in \X$. We deduce that $(\X,\sfd)$ contains a non-trivial rectifiable curve, which is a contradiction.
\end{proof}

The following classes of metric spaces are quasi-selfsimilar in the sense of Definition \ref{defin:quasi-selfsimilar}, so Proposition \ref{prop:snowflake_quasi-selfsimilar} applies to them: selfsimilar fractals (e.g. the polygaskets of \cite{TysonWu2005}), boundaries at infinity of cocompact Gromov-hyperbolic spaces (see \cite{Sullivan82, Kleiner2006, Cavallucci2025}), Julia
sets of expanding rational fractions (see \cite{Sullivan82}) and some iterated graph systems as defined in \cite{ErikssonBiqueAnttila2024}. In the last class, we notice that not all of the spaces considered in \cite{ErikssonBiqueAnttila2024} are quasigeodesic. For them, Proposition \ref{prop:snowflake_quasi-selfsimilar} becomes not trivial.

Finally, we notice that being biLipschitz equivalent to a snowflake is stable with respect to the operation of products of metric spaces.

 \begin{proposition}[Stability with respect to products]
 \label{prop:products}
     Let $(\X,\sfd), (\Y, \sfd')$ be two metric spaces. Then $(\X \times \Y, \sfd \times \sfd')$ is biLipschitz equivalent to a snowflake if and only if both $(\X,\sfd)$ and $(\Y, \sfd')$ are biLipschitz equivalent to a snowflake.
 \end{proposition}
 \begin{proof}
     Let $\omega$ be a non-principal ultrafilter, $\{r_n\}_{n\in \N}$ be a sequence of positive real numbers and $\{x_n\}_{n\in \N} \subseteq \X$, $\{y_n\}_{n\in \N} \subseteq \Y$. Then
     \begin{equation}
         \label{eq:ultralimit_product}
         \olim (\X\times \Y, r_n(\sfd \times \sfd'), (x_n,y_n)) \cong \olim (\X, r_n\sfd, x_n) \times \olim (\Y, r_n\sfd', y_n),
     \end{equation}
     by \cite[Lemma 4.11]{CavallucciSambusetti2024}.
     
     We prove the only if part. Suppose, without loss of generality, that $(\X,\sfd)$ is not biLipschitz equivalent to a snowflake. Theorem \ref{thm:characterization_snowflakes} implies that there exists an $\omega$-limit of $(\X,\sfd)$ that contains a non-trivial rectifiable curve. By \eqref{eq:ultralimit_product} we can find an $\omega$-limit of $(\X\times \Y, \sfd \times \sfd')$ that contains a non-trivial rectifiable curve, hence it is not biLipschitz equivalent to a snowflake by Theorem \ref{thm:characterization_snowflakes}.

     On the other hand, assume that the product is not biLipschitz equivalent to a snowflake. We can find an $\omega$-limit of the product that contains a non-trivial geodesic, by Theorem \ref{thm:characterization_snowflakes}. Every non-trivial geodesic in a product metric space projects to a reparametrized geodesic in each factor, see \cite[Proposition I.5.3]{BridsonHaefliger}, and the projection in one of the two factors must be non-trivial. Therefore, \eqref{eq:ultralimit_product} implies that one $\omega$-limit of $(\X,\sfd)$ or $(\Y,\sfd')$ contains a non-trivial geodesic. Hence, the corresponding factor is not biLipschitz equivalent to a snowflake, by Theorem \ref{thm:characterization_snowflakes}. In other words, if both factors are biLipschitz equivalent to a snowflake, then the product is.
 \end{proof}

Proposition \ref{prop:products} does not extend to the case of infinitely many products.

\begin{example}[Infinite products of snowflakes]
Let $\{\gamma_n\}_{n\in \N}$ be an increasing sequence of positive real numbers that converge to $1$. Consider the space $\X = \{\{x_n\}_{n\in \N}\,:\, \sum_{n\in \N} \vert x_n \vert^{2\gamma_n} < \infty\} \subseteq \ell^2$. Define the function $\sfd \colon \X \times \X \to \R$, $(\{x_n\}_{n\in \N}, \{y_n\}_{n\in \N}) \mapsto \left(\sum_{n\in \N} \vert x_n - y_n \vert^{2\gamma_n}\right)^{\frac{1}{2}}$. Then $(\X,\sfd)$ is a complete metric space. We claim it is not biLipschitz equivalent to a snowflake. Indeed, for every $k\in \N$ and $\alpha>0$ we can find $k$-points in a $n$-th coordinate direction, for $n$ sufficiently large, that are equally spaced and do not satisfy \eqref{eq:SRA}. The conclusion follows by Theorem \ref{thm:characterization_snowflakes}.
\end{example}

\bibliographystyle{alpha}
\bibliography{biblio}

\end{document}